\documentclass[twoside, 11pt]{article}
\usepackage{amssymb, amsmath, mathrsfs, amsthm}
\usepackage{graphicx}
\usepackage{color}
\usepackage[top=3cm, bottom=2cm, left=2cm, right=2cm]{geometry}
\usepackage{float, caption, subcaption}

\input{epsf}

\newcommand{\bd}{\begin{description}}
\newcommand{\ed}{\end{description}}
\newcommand{\bi}{\begin{itemize}}
\newcommand{\ei}{\end{itemize}}
\newcommand{\be}{\begin{enumerate}}
\newcommand{\ee}{\end{enumerate}}
\newcommand{\beq}{\begin{equation}}
\newcommand{\eeq}{\end{equation}}
\newcommand{\beqs}{\begin{eqnarray*}}
\newcommand{\eeqs}{\end{eqnarray*}}

\definecolor{DarkGreen}{rgb}{0.2, 0.6, 0.3}

\newtheorem{theorem}{Theorem}[section]

\newtheorem{lemma}{Lemma}[section]

\newtheorem{claim}{Claim}
\newtheorem{subclaim}{Subclaim}

\begin{document}
\title{\textbf{Euclidean Gallai-Ramsey Theory}  }

\author{Yaping Mao\footnote{Faculty of Environment and Information
Sciences, Yokohama National University, 79-2 Tokiwadai, Hodogaya-ku,
Yokohama 240-8501, Japan. {\tt mao-yaping-ht@ynu.ac.jp} and {\tt ozeki-kenta-xr@ynu.ac.jp}} , \ \ Kenta
Ozeki\footnotemark[1] \footnote{
This  work  was  supported by JSPSKAKENHI Grant Numbers 18K03391, 19H01803, 20H05795
}, \ \ Zhao
Wang\footnote{College of Science, China Jiliang University, Hangzhou
310018, China. {\tt wangzhao@mail.bnu.edu.cn}}~\footnote{This  work  was supported by the National
Science Foundation of China (No. 12061059).}}
\date{}
\maketitle

\begin{abstract}
In this paper, we introduce Euclidean Gallai-Ramsey theory,
by combining Euclidean Ramsey theory and Gallai-Ramsey theory on graphs.
More precisely,
we consider the following problem:
\begin{quote}
For an integer $r$ and configurations $K$ and $K'$,
does there exist an integer $n_0$
such that
for any $r$-coloring of the points of $n$-dimensional Euclidean space
with $n \geq n_0$,
there is a monochromatic configuration congruent to $K$
or
a rainbow configuration congruent to $K'$?
\end{quote}
In particular,
we give a bound on $n_0$ for some configurations $K$ and $K'$,
such as triangles and rectangles.
Those are extensions
of ordinary Euclidean Ramsey theory
where the purpose is to find a monochromatic configuration.
\\[2mm]
\\[2mm]
\end{abstract}

\section{Introduction}
\emph{Ramsey theory} is classical but still plays an important role in combinatorics.
Roughly speaking,
Ramsey theory aims to find certain monochromatic structures
for any coloring of a huge object.
The most famous theorem on Ramsey theory is about graphs,
and states that
for any integer $r$ and any graph $G$,
there exists an integer $n_0$ such that
for any $n \geq n_0$,
any (not necessarily proper) $r$-coloring of the edges of the complete graph $K_n$,
there is a monochromatic subgraph isomorphic to $G$.
We refer the reader
to a dynamic survey \cite{Radziszowski}
and a book \cite{GrahamRothschildSpencer} for Ramsey theory on graphs.
Ramsey theory has been extended in many areas,
such as Number theory, Set theory and so on.

In this paper, we focus on Ramsey theory of Euclidean space,
so called \emph{Euclidean Ramsey theory}.
We denote by $\mathbb{E}^n$ the $n$-dimensional Euclidean space.
A \emph{configuration} is
a finite set of points in an Euclidean space.
Two configurations $A=\{a_1,\ldots,a_t\}$ and $B=\{b_1,\ldots,b_t\}$
(not necessarily in the same space)
are
\emph{congruent} if
there exists a bijection $\varphi: \{1,2, \dots ,t\} \to \{1,2, \dots ,t\}$
such that $d(a_i,a_j)=d(b_{\varphi(i)},b_{\varphi(j)})$ for $1\leq
i<j\leq t$, where $d$ is the distance in the Euclidean space.
For a positive integer $r$,
we call an $r$-coloring of points of the $n$-dimensional Euclidean space
simply an \emph{$r$-coloring of $\mathbb{E}^n$}.
A configuration in $r$-colored $\mathbb{E}^n$
is said to be \emph{monochromatic}
if all the points in the configuration are colored by
the same color.
Euclidean Ramsey theory considers
the problems that
for any positive integer $r$ and any configuration $K$,
does there exist an integer $n_0$ such that
for any $n \geq n_0$ and any $r$-coloring of $\mathbb{E}^n$,
there exists a monochromatic configuration congruent to $K$?
Erd\"{o}s, Graham, Montgomery, Rothschild, Spencer, and
Straus \cite{EGMRSS1973} in 1973 first
investigated this problem.
We refer the reader to survey papers \cite{GrahamTressler, GrahamOldNew, 
GrahamRecentTrend}.

Recently, Gallai-Ramsey theory, which is a variation of Ramsey theory,
is widely studied.
This aims to find not only monochromatic structures
but also so-called rainbow ones
where any two elements have distinct colors.
After the work by Gallai \cite{MR0221974} on a rainbow triangle in $r$-colored complete graphs
(see also \cite{MR1464337, MR2063371}),
this study is called Gallai-Ramsey theory.
We refer to a dynamic survey \cite{FMO14} and some recent papers
\cite{FM, LB, MSonline} of Gallai-Ramsey numbers.

The purpose of this paper is to introduce a new concept,
\emph{Euclidean Gallai-Ramsey Theory}
by combining these two studies on Ramsey theory.
More precisely,
we consider the following problem,
where
a \emph{rainbow} configuration in $r$-colored $\mathbb{E}^n$
is one with all points colored by distinct colors:
\begin{quote}
For a positive integer $r$ and configurations $K$ and $K'$,
does there exist an integer $n_0$
such that
for any $r$-coloring of the points of $n$-dimensional Euclidean space
with $n \geq n_0$,
there is a monochromatic configuration congruent to $K$
or
a rainbow configuration congruent to $K'$?
\end{quote}
We consider this problem
for some basic configurations,
such as triangles and rectangles.

First,
we consider the case where both $K$ and $K'$ are $T$,
where $T$ is a triangle with angles $30$, $60$ and $90$ degrees
and with hypotenuse of unit length,
B\'{o}na \cite{Bona} proved that
for any $3$-coloring of $\mathbb{E}^3$,
there exists a monochromatic configuration congruent to $T$,
and also
there exists a $12$-coloring of $\mathbb{E}^3$
in which there is no monochromatic configuration congruent to $T$.
In contrast with this result,
we show that
there always exist a monochromatic configuration congruent to $T$
or
a rainbow configuration congruent to $T$
in any $r$-coloring of $\mathbb{E}^3$ with $r \geq 3$.
We prove it in Section \ref{sec:tri}.
\begin{theorem}\label{th2-7}
For any positive integer $r$
and any $r$-coloring of $\mathbb{E}^3$,
there is a monochromatic configuration congruent to $T$
or
a rainbow configuration congruent to $T$.
\end{theorem}

Next,
we consider the case where both $K$ and $K'$ is a rectangle.
Erd\"{o}s, Graham, Rothschild, Montgomery, Spencer, and Straus \cite{EGMRSS1973}
proved that for any rectangle $Q$ and any $2$-coloring of $\mathbb{E}^8$,
there exists a monochromatic rectangle congruent to $Q$.
T\'{o}th \cite{Toth1996} improved this theorem
by showing that
the same statement holds even for $\mathbb{E}^5$.
Cantwell \cite{CantwellFinite1996}
showed that $\mathbb{E}^4$ suffices when a rectangle $Q$ is a square.
T\'{o}th \cite{Toth1996} also investigated colorings with many
colors and proved that for any $r$-coloring of $\mathbb{E}^{r^2 +o(r^2)}$,
there exists a monochromatic configuration congruent to any rectangle.
We extend his result to the following theorem,
which will be proven in Section \ref{sec:rect}.
\begin{theorem}
\label{thm:rectangle}
Let $r$ be a positive integer and $Q$ be a rectangle.
Then,
for any $r$-coloring of $\mathbb{E}^{13r+4}$,
there exists a monochromatic rectangle congruent to $Q$
or a rainbow rectangle congruent to $Q$.
\end{theorem}
Thus,
compared with the result by T\'{o}th,
an Euclidean space with much fewer dimension suffices
to find certain configurations in any $r$-coloring.
On the other hand,
for a rectangle $Q$ with side length $a$ and $b$ with $a \leq b \leq \sqrt{3} a$,
we show that the same statement does not hold for $r \geq 3$
if we replace $\mathbb{E}^{13r+4}$ with $\mathbb{E}^2$.
We prove this in Section \ref{sec:lowerbound}.
\begin{theorem}
\label{thm:lowerbound}
For a positive integer $r \geq 3$
and a rectangle $Q$ with side length $a$ and $b$ with $a \leq b \leq \sqrt{3} a$,
there exists an $r$-coloring of $\mathbb{E}^{2}$
that contains
neither a monochromatic rectangle congruent to $Q$
nor a rainbow rectangle congruent to $Q$.
\end{theorem}

We also consider the case where
$K$ is a particular triangle or a square,
and $K'$ is an equilateral triangle.
More precisely,
we prove the following theorem in Section \ref{sec:others},
where an $(a,b,c)$-triangle
is one with side length $a, b$ and $c$:
\begin{theorem}\label{th-2-16}
For a positive integer $r$,
the following hold:

\noindent
$(1)$ For any $r$-coloring of $\mathbb{E}^{r+4}$,
there exists a monochromatic unit square or a rainbow equilateral triangle.

\noindent
$(2)$ For any $r$-coloring of $\mathbb{E}^{r+3}$,
there exists a monochromatic $(1, 1, \sqrt{2})$-triangle
or a rainbow equilateral triangle.
\end{theorem}

\section{Monochromatic triangle versus rainbow triangle}
\label{sec:tri}

In this section,
we prove Theorem \ref{th2-7},
which states that
for any integer $r \geq 3$
and any $r$-coloring of $\mathbb{E}^3$,
there is a monochromatic configuration congruent to $T$
or
a rainbow configuration congruent to $T$.
Recall that
$T$ is a triangle with angles $30$, $60$ and $90$ degrees
and with hypotenuse of unit length.
For the proof, we use the following theorem.
\begin{theorem}\label{th2-1}
\label{unitline_thm}
Let $r$ be a positive integer and let $d$ be a positive real number,
Then, for any $r$-coloring of $\mathbb{E}^{2}$,
if at least two colors are used,
then there exists a rainbow segment (that is, two points with distinct colors)
of length $d$.
\end{theorem}
\begin{proof}
Assume, to the contrary, that there exists no rainbow segment
of length $d$ in $\mathbb{E}^{2}$.
Note that there exists a rainbow segment $CD$ in $\mathbb{E}^{2}$,
and let $d'$ be its length.
We take such a rainbow segment $CD$ so that
the length $d'$ is as small as possible.
Without loss of generality, assume that $C$ is red and $D$ is blue.
If $d'\leq 2d$,
then we can find a perpendicular bisector of $CD$, say $\ell$.
There exists a point $E$ in $\ell$ such that $|EC| = |ED| = d$.
Then at least one of $EC$ or $ED$ is a rainbow segment of length $d$,
a contradiction.

Suppose $d' > 2d$.
Let $C'$ be the point on the segment $CD$ with distance $d$ from $C$.
To avoid a rainbow segment of length $d$, $C'$ is colored by red.
Then $C'D$ is a rainbow segment of length shorter than $CD$,
contradicting the choice of $CD$.
This completes the proof.
\end{proof}

\begin{proof}[Proof of Theorem \ref{th2-7}]

We prove that for any $r$-coloring of $\mathbb{E}^3$,
there exists a monochromatic rectangle congruent to $T$
or a rainbow rectangle congruent to $T$.
If only one color is used,
then clearly there exists a monochromatic rectangle congruent to $T$.
Thus, we may assume that at least two colors are used,
By Theorem \ref{unitline_thm},
there is a rainbow segment $AB$ of length $\sqrt{3}/2$.
Take an equilateral triangle $ABC$ of side length $\sqrt{3}/2$ that uses $AB$ as one side.
By symmetry, we may assume that $A$ is red and  $B$ is blue.
In addition, we may further assume that $C$ is not blue,
since if $C$ is blue, then we can switch the
roles of $A$ and $B$, and also roles of their colors.

Let $\ell$ be the line such that $\ell$ is perpendicular to the
regular triangle $ABC$ and $\ell$ passes through $A$. Let
$A_1A_2,\ldots A_6$ be a regular hexagon satisfying the following
conditions.
\begin{itemize}
\item[] $(1)$ The center of the regular hexagon $A_1A_2\ldots A_6$ is $A$ and each side is of length $1/2$.
(This implies that $|AA_i| = 1/2$ for $i = 1,2,3,4,5,6$.)
\item[] $(2)$ The regular hexagon $A_1A_2\ldots A_6$ is
perpendicular to the line $AB$. (This implies that the line $AA_i$
for $1 \leq i \leq 6$ is perpendicular to $AB$.)
\item[] $(3)$ Both $A_1$ and $ A_4$ are contained in $\ell$.
(This, together with the choice of $\ell$ implies that the lines
$AA_1$ and $AA_4$ are both perpendicular to $AC$.)
\end{itemize}
To avoid a rainbow triangle congruent to $T$
formed by $A, B$ and $A_i$ for $1 \leq i \leq 6$,
we see that each $A_i$ is either red or blue.
Suppose that $A_1$ and $A_4$ are the same color.
We now assume that $A_1$ and $A_4$ are both red,
but the following argument also works even in the case that $A_1$ and $A_4$ are both blue.
Since the triangle congruent to $T$
formed by $A_1, A_4$ and $A_j$ for $j = 2,3,5,6$ is not monochromatic,
we see that $A_j$ is blue for any $j= 2,3,5,6$,
but then $A_2, A_3$ and $A_5$ form a monochromatic triangle congruent to $T$.

Thus,
we may assume that $A_1$ is red and $A_4$ is blue.
Recall that $C$ is not blue.
Thus, if $C$ is red, then $A,A_1$ and $C$ form
a monochromatic triangle congruent to $T$:
Otherwise, that is, if $C$ is neither red nor blue,
then $A,A_4$ and $C$ form a rainbow triangle congruent to $T$.
This completes the proof of Theorem \ref{th2-7}.
\end{proof}

\section{Monochromatic rectangle versus rainbow rectangle}

In this section,
we consider to find a monochromatic rectangle congruent to $Q$
or a monochromatic rectangle congruent to $Q$.
in an $r$-coloring of an Euclidean space,
where $Q$ is a given rectangle,
and prove Theorems \ref{thm:rectangle} and \ref{thm:lowerbound}.

\subsection{Proof of Theorem \ref{thm:rectangle}}
\label{sec:rect}

In this section,
we prove Theorem \ref{thm:rectangle}.
Recall that
T\'{o}th \cite{Toth1996} proved that
for any positive integer $r$ and any $r$-coloring of $\mathbb{E}^{r^2 + o(r^2)}$,
there exists a monochromatic rectangle congruent to a given rectangle.
More precisely,
he proved the following theorem.
\begin{theorem}[T\'{o}th \cite{Toth1996}]
\label{thm:Toth}
For any positive integer $r$,
there exist two integers $n = o(r^2)$ and $m = r^2 + o(r^2)$
such that
for any $r$-coloring of $\mathbb{E}^{n+m}$
and any rectangle $Q$,
there exists a monochromatic rectangle congruent to $Q$.
\end{theorem}
We employ the idea of the proof of Theorem \ref{thm:Toth} in \cite{Toth1996}.
To explain that, we need some terminology.
For two positive integers $n$ and $m$,
the set $\{(i,j) \in \mathbb{E}^2 : 1 \leq i \leq n, \ 1 \leq j \leq m\}$
of points is called the \emph{$(n \times m)$-gird}.
A \emph{rectangle} of the $(n \times m)$-gird
is a configuration consisting of four points $(i,j), (i',j), (i,j')$ and $(i',j')$
for some $1 \leq i < i' \leq n, \ 1 \leq j < j' \leq m$.
Suppose that the points in the $(n \times m)$-gird is $r$-colored where $r \geq 1$.
Then a rectangle in the grid is
said to be \emph{monochromatic}
if the four points forming the rectangle are colored by the same color.
Similarly,
a rectangle is \emph{rainbow}
if the four points forming the rectangle are colored by distinct colors.

In order to prove Theorem \ref{thm:Toth},
T\'{o}th proved the following lemma (though he did not clearly state that).
\begin{lemma}[T\'{o}th \cite{Toth1996}]
\label{lemma:gridmono}
For any positive integer $r$,
there exist two integers $n = o(r^2)$ and $m = r^2 + o(r^2)$
such that
for any $r$-coloring of the $(n \times m)$-grid,
there exists a monochromatic rectangle.
\end{lemma}
To adapt the idea by T\'{o}th,
we will prove the following lemma.
\begin{lemma}
\label{lemma:grid}
For any positive integer $r$
and for any $r$-coloring of the $\left((2r+5) \times (11r+1)\right)$-grid,
there exists a monochromatic rectangle
or a rainbow rectangle.
\end{lemma}

\noindent
\textbf{Remark:}
Lemmas \ref{lemma:gridmono} and \ref{lemma:grid} are
related to \emph{bipartite Ramsey theory} and \emph{bipartite Gallai-Ramsey theory},
respectively.
Consider the bijection that maps the point $(i,j)$ of the $(n \times m)$-grid
to the edge connecting $i$th vertex in one bipartite set
and $j$th vertex in the other bipartite set of the complete bipartite graph $K_{n,m}$.
By this bijection,
Lemma \ref{lemma:gridmono} is shown to be equivalent to the statement that
for any $r$-coloring of the edges of
the complete bipartite graph $K_{n, m}$,
where $n$ and $m$ are some integers with $n  =o(r^2)$ and $m = r + o(r^2)$,
there exists a monochromatic subgraph isomorphic to $K_{2,2}$.
With this context,
Lemma \ref{lemma:gridmono} has been extended
in several direction \cite{BLS19, Carnielli99, HH98, HJ14, WLL21},
where such a study is called bipartite Ramsey theory.
Similarly,
Lemma \ref{lemma:grid} is equivalent to the statement that
for any $r$-coloring of the edges of the complete bipartite graph $K_{2r+5, 11r+1}$,
there exists a monochromatic subgraph isomorphic to $K_{2,2}$
or a rainbow subgraph isomorphic to $K_{2,2}$.
Some results on bipartite Gallai-Ramsey numbers can be found in \cite{LWL19}.\\

Before proving  Lemma \ref{lemma:grid},
we show Theorem \ref{thm:rectangle} assuming Lemma \ref{lemma:grid}.
Note that the proof of Theorem \ref{thm:rectangle} is almost same as the proof by T\'{o}th \cite{Toth1996},
while we need new ideas to prove Lemma \ref{lemma:grid}.

\begin{proof}[Proof of Theorem \ref{thm:rectangle}]
Let $Q$ be a given rectangle,
and let $a$ and $b$ be the two sides of $Q$.
For a fixed $r$, we consider an $r$-coloring of $\mathbb{E}^{13r+4}$.
Let $S$ and $P$ be two complementary orthogonal subspaces of $\mathbb{E}^{13r+4}$
such that their dimension is $2r+4$ and $11r$, respectively,
and let $A_{1,1}$ denote the origin point,
which is the intersection of $S$ and $P$.
Let $A_{1,1}A_{2,1}\ldots A_{2r+5,1}$ be a regular simplex of side $a$ in $S$,
and let $M_1=A_{1,1}A_{1,2}\ldots A_{1,11r+1}$ be a regular simplex of side $b$ in $P$.
For any $1<i \leq 2r+5$ and $1 < j\leq 11r+1$,
define the point $A_{i,j}$ as the image of $A_{1,j}$ under a
translation of $M_1$ taking $A_{1,1}$ into $A_{i,1}$.

We can regard the point $A_{i,j}$ of $\mathbb{E}^{13r+4}$
with $1 \leq i \leq 2r+5$ and $1 \leq j\leq 11r+1$
as the point $(i,j)$ of the $\left((2r+5) \times (11r+1)\right)$-grid,
which is $r$-colored by the $r$-coloring of $\mathbb{E}^{13r+4}$.
By Lemma \ref{lemma:grid},
there exists a monochromatic rectangle
or a rainbow rectangle,
and let $(i,j), (i',j), (i,j')$ and $(i',j')$
be the four points forming the rectangle.
By the construction,
the lines $A_{i,j}A_{i',j}$ and $A_{i,j'}A_{i',j'}$ are of length $a$,
the lines $A_{i,j}A_{i,j'}$ and $A_{i',j}A_{i',j'}$ are of length $b$,
and
the two lines $A_{i,j}A_{i,j'}$ and $A_{i,j}A_{i',j}$ cross at right angle.
Therefore,
the four points $A_{i,j}, A_{i',j}, A_{i,j'}$ and $A_{i',j'}$
form a monochromatic rectangle congruent to $Q$
or a rainbow rectangle congruent to $Q$.
\end{proof}

\begin{proof}[Proof of Lemma \ref{lemma:grid}]
We only prove the case $r \geq 4$,
since the case when $r \leq 3$ was done in \cite{Toth1996}.

Suppose contrary that
there is neither a monochromatic rectangle nor a rainbow rectangle.
We use the set $\{1,2, \dots , r\}$ for the colors of the $r$-coloring
of the $\left((2r+5) \times (11r+1)\right)$-grid.
For $1 \leq t \leq r$,
let
$$B_{1}^t = \{(1,j) \,|\, 1 \leq j \leq 11r+1, \ \text{$(1,j)$ is colored by $t$}\}.$$
We may assume $|B_{1}^1|\geq |B_{1}^2|\geq \dots \geq |B_{1}^r|$.
Since $\sum_{t=1}^r|B_{1}^t| = 11r+1$, it follows that $|B_{1}^1|\geq 12$.
For each $i$ with $2 \leq i \leq 2r+5$
and each $t$ with $1 \leq t \leq r$,
let
$$B_{i}^t = \{(i,j) \,|\, 1 \leq j \leq 11r+1, \ (1,j) \in B_1^t\}.$$
To avoid a monochromatic rectangle in $B_1^1 \cup B_i^1$ with $2 \leq i \leq 2r+5$,
there are at most one point with color $1$ in $B_{i}^1$.
Let $s$ be the number of point sets $B_{i}^1$
with $2\leq i \leq 2r+5$
such that there are at least $3$ colors except
the color $1$ on the points of $B_{i}^1$.
We first prove the following claim.

\begin{claim}\label{Claim3}
$s\geq 5$.
\end{claim}

\begin{proof}
Assume, to the contrary, that $s \leq 4$.
Then there exist $2r$ sets,
without loss of generality, say $B_{2}^1,B_{3}^1,\dots,B_{2r+1}^1$,
such that there are at most $2$ colors except the color $1$ on the
points of $B_{i}^1$.
Let $\mathcal{B}=\{B_{i}^1\,|\, 2 \leq i\leq 2r+1\}$.
We have the following subclaim.

\begin{subclaim}\label{Subclaim1}
There are no two point sets $B_{a}^1,B_{b}^1 \in \mathcal{B}$
such that each of them has two colors except the color $1$,
and the union $B_{a}^1 \cup B_{b}^1$ has four colors except $1$.
\end{subclaim}
\begin{proof}
Suppose contrary that such two point sets $B_{a}^1$ and $B_{b}^1$ exist.
For simplicity,
we let the two colors on the points of $B_{a}^1$ except the color $1$ are $2$ and $3$,
and those for $B_{b}^1$ are $4$ and $5$.
Let $X_a,Y_a$ be the set of points with colors $2,3$ in $B_{a}^1$, respectively,
let $X_b = \{(b,j) \,|\, 1 \leq j \leq 11r+1, \ (a,j) \in X_{a}\}$,
and let $Y_b = \{(b,j) \,|\, 1 \leq j \leq 11r+1, \ (a,j) \in Y_{a}\}$.
Since there is at most one point with color $1$ in $B_{a}^1$,
it follows that $|B_{a}^1|=|X_a|+|Y_a|$ or $|B_{a}^1|=|X_a|+|Y_a|+1$,
where $|X_a|\geq 1$ and $|Y_a|\geq 1$.
Then there exists one point, say $(b,j)$, in $X_b \cup Y_b$,
such that $(b,j)$ is colored by either $4$ or $5$,
say $4$ without loss of generality.
By symmetry,
we assume that $(a,j) \in X_a$,
that is, $(a,j)$ is colored by $2$.
To avoid a rainbow rectangle of colors $2,3,4$ and $5$,
all points in $Y_b$ must be colored by $4$.
To avoid a rainbow rectangle of colors $2,3,4$ and $5$,
all points in $X_b$ must be colored by $4$.
If $|B_{a}^1|=|X_a|+|Y_a|$,
then all points in $B_{b}^1$ are colored by $4$,
which contradicts to the fact that $B_{b}^1$ has two colors except $1$.
If $|B_{a}^1|=|X_a|+|Y_a|+1$,
then there is an integer $j'$ with $1 \leq j' \leq 11r+1$
such that $(a,j') \in B_{a}^1$ is colored by $1$
and $(b,j')$ is colored by $5$,
which together with $(a,j)$ and $(b,j)$ form a rainbow rectangle of colors $1,2,4$ and $5$,
 a contradiction.
\end{proof}

For $2 \leq t \leq r$,
we say that a point set $B_{i}^1 \in \mathcal{B}$ is of \emph{type $[t]$}
if $B_{i}^1$ contains only the color $t$ except the color $1$.
Since all points in $B_{1}^1$ are colored by $1$,
the possible types are $[2],[3],\dots,[r]$.
Let $\mathcal{S}=\{[t]\,|\,2\leq t\leq r,
\text{ there exists a point set in $\mathcal{B}$ of type $[t]$}
\}$.

\begin{subclaim}\label{Subclaim2}
For each $t$ with $2\leq t \leq r$,
there is at most one point set in $\mathcal{B}$
of type $[t]$.
\end{subclaim}
\begin{proof}
Assume, to the contrary, that there are two point sets in $\mathcal{B}$
of type $[t]$, say $B_{a}^1$ and $B_{b}^1$.
Recall that $|B_{a}^1|=|B_{b}^1|=|B_{1}^1|\geq 12$.
Then, there exist two points $(a,j)$ and $(a,j')$ in $B_{a}^1$ with color $t$
such that
$(b,j)$ and $(b,j')$ are also colored by $t$,
and hence there is a monochromatic rectangle with color $t$,
a contradiction.
\end{proof}

For $2\leq t<t'\leq r$,
we say that a point set $B_{i}^1 \in \mathcal{B}$ is of \emph{type $[t,t']$}
if $B_{i}^1$ contains only the colors $t$ and $t'$ except the color $1$.
Let $\mathcal{T}=\{[t,t'] \,|\,2\leq t<t'\leq r,
\text{ there exists a point set in $\mathcal{B}$ of type $[t,t'])$}
\}$.

\begin{subclaim}\label{Subclaim3}
For $2\leq t<t'\leq r$,
there are at most two point sets in $\mathcal{B}$
of type $[t,t']$.
\end{subclaim}
\begin{proof}
Assume, to the contrary, that
there are at least three point sets in $\mathcal{B}$
of type $[t,t']$,
say $B_{a}^1,B_{b}^1$ and $B_{c}^1$.
Since
$|B_{a}^1| = |B_{1}^1|\geq 12$,
there exist at least six points in $B_{a}^1$,
say $(a,j_1), \dots , (a,j_6)$ with color $t$.
To avoid a monochromatic rectangle of color $t$,
there is at most one point in $\{(b,j_k) \,|\, 1\leq k\leq 6\}$ with color $t$.
Since there is at most one point in $\{(b,j_k) \,|\,1\leq k\leq 6\}$ with color $1$,
at least four points in $\{(b,j_k) \,|\,1\leq k\leq 6\}$ are colored by $t'$.
The same is true for the points in $\{(c,j_k) \,|\,1\leq k\leq 6\}$.
So there exist two points $(b,j_k)$ and $(b,j_{k'})$
in $\{(b,j_k) \,|\,1\leq k\leq 6\}$ with color $t'$
such that
the two points $(c,j_k), (c,j_{k'})$ are also colored by $t'$,
which form a monochromatic rectangle,
a contradiction.
\end{proof}

If $\mathcal{T}=\emptyset$,
then each point set in $\mathcal{B}$ is of type in $\mathcal{S}$.
Since $|\mathcal{B}|=2r$ and $|\mathcal{S}| \leq r-1$,
there exist two point sets in $\mathcal{B}$ of the same type,
but this contradicts Subclaim \ref{Subclaim2}.
Thus, without loss of generality, we may assume $[2,3]\in \mathcal{T}$.
We have the following subclaims.

\begin{subclaim}\label{Subclaim5}
We have $[2] \notin \mathcal{S}$
or $[3] \notin \mathcal{S}$.

\end{subclaim}
\begin{proof}
Since $[2,3] \in \mathcal{T}$,
there exists a point set $B_{i}^1$ in $\mathcal{B}$ of type $[2,3]$.
By symmetry, we may assume that the number of points in $B_{i}^1$ with color $2$
is large than the number of points in $B_{i}^1$ with color $3$.
If there exists a point set $B_{i'}^1$ in $\mathcal{B}$ of type $[2]$,
then there exists a monochromatic rectangle in $B_{i}^1 \cup B_{i'}^1$ with color $2$,
a contradiction.
\end{proof}

\begin{subclaim}\label{Subclaim10}
For $4 \leq t \leq r$,
there are at most two point sets in $\mathcal{B}$ of type
either $[t]$ or $[2,t]$
and
there are at most two point sets in $\mathcal{B}$ of type
either $[t]$ or $[3,t]$.
\end{subclaim}
\begin{proof}
Assume, to the contrary, that
for some color $t$ with $4 \leq t \leq r$,
there are three point sets in $\mathcal{B}$
of type either $[t]$ or $[2,t]$,
say $B_{a}^1,B_{b}^1$ and $B_{c}^1$.
By Subclaims \ref{Subclaim2} and \ref{Subclaim3} and the symmetry,
we may assume that $B_{a}^1,B_{b}^1$ and $B_{c}^1$ are of type $[t], [2,t]$ and $[2,t]$, respectively.
To avoid a monochromatic rectangle with color $t$ in $B_{a}^1 \cup B_{b}^1$,
there are at most two points in $B_{b}^1$ with color $t$.
Since $B_{b}^1$ contains at most one point with color $1$,
$B_{b}^1$ contains at least $|B_{b}^1| - 3$ points with color $2$.
Similarly,
$B_{c}^1$ contains at least $|B_{c}^1| - 3$ points with color $2$.
Since $|B_{b}^1| = |B_{c}^1| = |B_{1}^1| \geq 12$,
there are at least two points $(b,j), (b,j') \in B_{b}^1$ with color $2$
such that
$(c,j), (c,j')$ are also colored by $2$.
Then, those four points form a monochromatic rectangle with color $2$,
a contradiction.
This shows that
there are at most two point sets in $\mathcal{B}$ of type
either $[t]$ or $[2,t]$.
By symmetry,
we can obtain the second statement.
\end{proof}

For any $[t,t'] \in \mathcal{T}$, we have either $t=2$ or $t=3$,
since otherwise
the point set of type $[t,t']$ and
the point set of type $[2,3]$ contradicts Subclaim \ref{Subclaim1}.
Let $\mathcal{T}_2=\{[2,t] \in \mathcal{T} \,|\,4\leq t\leq r\}$
and $\mathcal{T}_3=\{[3,t] \in \mathcal{T} \,|\,4\leq t\leq r\}$.
By Subclaim \ref{Subclaim10},
$\mathcal{T}=\{[2,3]\} \cup \mathcal{T}_2\cup \mathcal{T}_3$.
We divide the rest of the proof into two cases
according to whether $\mathcal{T}_2\neq \emptyset$ and $\mathcal{T}_3\neq \emptyset$.
\\

\noindent
\textbf{Case (a): $\mathcal{T}_2\neq \emptyset$ and $\mathcal{T}_3\neq \emptyset$.}

Let $[2,t_0] \in \mathcal{T}_2$.
By Subclaim \ref{Subclaim1},
$[3,t] \notin \mathcal{T}_3$ for any $t \neq t_0$,
and hence
$[3,t_0]$ is the only element in $\mathcal{T}_3$.
Similarly,
we see that $[2,t_0]$ is the only element in $\mathcal{T}_2$.
Therefore,
$\mathcal{T}_2 = \{[2,t_0]\}$ and $\mathcal{T}_3 = \{[3,t_0]\}$.
By Subclaim \ref{Subclaim10},
there are at most four point sets in $\mathcal{B}$ of type either $[t_0], [2,t_0]$ or $[3,t_0]$.
By Subclaim \ref{Subclaim3}, there are at most two point sets in $\mathcal{B}$ of type $[2,3]$.
ws that
By Subclaims \ref{Subclaim2} and \ref{Subclaim5},
there exist at most one point set in $\mathcal{B}$
of type each of $\mathcal{S} \setminus \{[t_0]\}$
and $|\mathcal{S} \setminus \{[t_0]\}| \leq r-3$.
Since the types of point sets in $\mathcal{B}$ in this case are
contained in $\mathcal{S} \cup \{[2,3], [2,t_0], [3,t_0]\}$.
Therefore,
the number of point sets in $\mathcal{B}$ is at most $4+2+(r-3)=r+3$,
which is less than $2r=|\mathcal{B}|$ since $r \geq 4$,
a contradiction .
\\

\noindent
\textbf{Case (b): Either $\mathcal{T}_2 = \emptyset$ or $\mathcal{T}_3 = \emptyset$.}

By symmetry,
we may assume $\mathcal{T}_3 = \emptyset$.
By Subclaim \ref{Subclaim10}, there are at most $2(r-3)$ point sets in $\mathcal{B}$
of type in $\{[2,t] \,|\,4\leq t\leq r\}\cup \{[t]\,|\,4\leq t\leq r\}$.
By Subclaim \ref{Subclaim3}, there are at most two point sets in $\mathcal{B}$ of type $[2,3]$.
By Subclaims \ref{Subclaim2} and \ref{Subclaim5},
there is at most one point set in $\mathcal{B}$ of type either $[2]$ or $[3]$.
Thus, the number of point sets in $\mathcal{B}$ is
at most $2(r-3)+2+1=2r-3<2r = |\mathcal{B}|$, a contradiction.
This completes the proof of Claim \ref{Claim3}.
\end{proof}

By Claim \ref{Claim3}, there are at least five point sets
in $\{B_{i}^1 \,|\, 2\leq i\leq 2r+5\}$, say
$B_{2}^1,B_{3}^1, \dots , B_{6}^1$ by symmetry,
such that there are at least $3$ colors except
the color $1$ on the points of $B_{k}^1$ with $2\leq k \leq 6$.
Using these point sets,
we next check the size of $B_{1}^2, \dots , B_{1}^r$.

\begin{claim}\label{Claim4}
$|B_{1}^2|\leq 2$.
\end{claim}
\begin{proof}
Assume, to the contrary, that $|B_{1}^2|\geq 3$.
Let $(1,a), (1,b), (1,c) \in B_{1}^2$.
Recall that the points $(1,a), (1,b)$ and $(1,c)$ are colored by $2$.
To avoid a monochromatic rectangle with color $2$,
for $2\leq k \leq 6$,
there is at most one point in $(k,a), (k,b)$ and $(k,c)$ with color $2$,
and
the others are all colored by $1$
in order to avoid a rainbow rectangle formed by a point in $B_{1}^1$,
a point in $B_{k}^1$,
a point in $\{(k,a), (k,b), (k,c)\}$,
and a point in $\{(1,a), (1,b), (1,c)\}$.
Then, there exist two integers $k$ and $k'$ with $2 \leq k < k' \leq 6$
such that the colors of
$(k,a),(k,b)$ and $(k,c)$
coincide with those of $(k',a),(k',b)$ and $(k',c)$,
respectively.
Hence, there is a monochromatic rectangle colored by $1$
with two points in $\{(k,a),(k,b), (k,c)\}$
and
two points in $\{(k',a),(k',b), (k',c)\}$,
a contradiction.
\end{proof}

\begin{claim}\label{Claim5}
$|B_{1}^3|\leq 1$.
\end{claim}
\begin{proof}
Assume, to the contrary, that $|B_{1}^3|\geq 2$.
Then $|B_{1}^2| \geq |B_{1}^3| \geq 2$.
By Claim \ref{Claim4}, $|B_{1}^3|=|B_{1}^2|=2$.
For $2 \leq k\leq 6$,
to avoid a monochromatic rectangle with color $2$ in
$B_{1}^2\cup B_{k}^2$,
there is at most one point in $B_{k}^2$ with color $2$,
and the other points are colored by $1$
in order to avoid a rainbow rectangle.
By the same reason, for $2\leq k\leq 6$,
there is at most one point in $B_{k}^3$ with color $3$,
and the other points are colored by $1$.
Then
for some $2 \leq k < k' \leq 6$
and some $(k,a) \in B_{k}^2$ and $(k,b) \in B_{k}^3$
such that
all of the four points
$(k,a), (k,b), (k',a), (k,b)$ are colored by $1$.
Thus,
they form a monochromatic rectangle colored by $1$,
a contradiction.
\end{proof}

By Claims \ref{Claim4} and \ref{Claim5}, we have $|B_{1}^2|\leq 2$
and $|B_{1}^3|\leq 1$. Then $|B_{1}^t|\leq 1$ for $4\leq t\leq r$,
and hence $|B_{1}^1| \geq (11r+1) - r = 10r+1$.
We now rename $C_1 = B_{1}^1$.
Recall that all points of $C_1 = B_{1}^1$ are colored by $1$.
By the symmetry of the first coordinate,
we have the following claim.
\begin{claim}\label{cla}
For each $i$ with $1\leq i \leq 2r+5$,
there exists a set $C_{i} \subseteq \{(i,j)\,|\, 1 \leq j \leq 11r+1\}$
such that $|C_{i}|\geq 10r+1$
and all points in $C_i$ are colored by the same color.
\end{claim}

Since 
there exist $2r+5$ sets $C_1, C_2, \dots , C_{2r+5}$ but only $r$ colors,
there exist two sets $C_a$ and $C_b$
such that all points in $C_a \cup C_{b}$ have the same color,
where $1\leq a < b \leq 2r+5$.
Note that $|C_{a}| \geq 10r+1$ and $|C_{b}| \geq 10r+1$.
Since $C_{a} \subseteq \{(a,j)\,|\, 1 \leq j \leq 11r+1\}$
and $C_{b} \subseteq \{(b,j)\,|\, 1 \leq j \leq 11r+1\}$,
there are at least $9r+1$ points $(a,j)$ in $C_{a}$
with $(b,j) \in C_b$.
Then,
there is a monochromatic rectangle in $C_a \cup C_b$,
a contradiction.
\end{proof}

\subsection{Proof of Theorem \ref{thm:lowerbound}}
\label{sec:lowerbound}

We prove Theorem \ref{thm:lowerbound},
which states that for a positive integer $r \geq 3$
and a rectangle $Q$ with side length $a$ and $b$ with $a \leq b \leq \sqrt{3} a$,
there exists an $r$-coloring of $\mathbb{E}^{2}$
that contains
neither a monochromatic rectangle congruent to $Q$
nor a rainbow rectangle congruent to $Q$.

\begin{proof}[Proof of Theorem \ref{thm:lowerbound}]
For an integer $i$,
let $X_i=\{(x,y)\in \mathbb{E}^2\,|\,ia\leq x<(i+1)a\}$.
Note that $\mathbb{E}^2$ is partitioned into $X_i$'s,
that is,
$\bigcup_{i \in \mathbb{Z}} X_i = \mathbb{E}^2$
and $X_i \cap X_{i'} = \emptyset$ for $i \neq i'$.
Consider the $r$-coloring of $\mathbb{E}^2$
such that the points in $X_i$ are colored by $i \pmod{r}$
for each integer $i$.
For any pair $X_i,X_j$ with $i \equiv j \pmod{r}$,
the distance between a point in $X_i$ and that in $X_j$ is larger than $(r-1)a > b$.
Thus,
if there exists a monochromatic rectangle congruent to $Q$,
then the four points must be contained in only one set $X_i$,
but this is impossible by the definition of $X_i$.
On the other hand,
if there exists a rainbow rectangle congruent to $Q$,
then its points are contained in distinct point sets,
say $X_{i_1},X_{i_2},X_{i_3},X_{i_4}$
with $i_1 < i_2 < i_3 < i_4$.
Since $b\leq \sqrt{3} a$,
it follows that the distance of any two points in $Q$ is at most $2a$,
which contradicts to the fact that the distance of a point in $X_{i_1}$
and a point in $X_{i_4}$ is larger than $2a$.
\end{proof}

\section{Monochromatic triangle or rectangle versus rainbow equilateral triangle}
\label{sec:others}

In this section,
we prove Theorem \ref{th-2-16},
which states that
$(1)$ for any $r$-coloring of $\mathbb{E}^{r+4}$,
there exists a monochromatic unit square or a rainbow equilateral triangle,
and $(2)$ for any $r$-coloring of $\mathbb{E}^{r+3}$,
there exists a monochromatic $(1, 1, \sqrt{2})$-triangle
or a rainbow equilateral triangle.
Recall that an $(a,b,c)$-triangle
is one with side length $a, b$ and $c$.

For the proof,
we use some results on Gallai-Ramsey theory on graphs.
For two graphs $G$ and $H$,
the \emph{$r$-colored Gallai-Ramsey number of $G$ and $H$},
denoted by ${\rm gr}_r(G: H)$,
is the minimum integer $n$ such that
for any $r$-coloring of the edges of the complete graph $K_n$,
there exists a rainbow subgraph isomorphic to $G$
or
a monochromatic subgraph isomorphic to $H$.
We denote the cycle (resp.~the path) of order four by $C_4$ (resp.~$P_4$).
Faudree, Gould, Jacobson, and Magnant
\cite{FaudreeGouldJacobsonMagnant} obtained the following results.
\begin{theorem}{\upshape \cite{FaudreeGouldJacobsonMagnant}}\label{th-2-15}
The following holds:

\noindent
$(1)$ For $r\geq 3$, $\operatorname{gr}_r(K_3:C_4)=r+4$.

\noindent
$(2)$ For $r\geq 3$, $\operatorname{gr}_r(K_3:P_4)=r+3$.
\end{theorem}

Now, we are ready to prove Theorem \ref{th-2-16}.
\begin{proof}[Proof of Theorem \ref{th-2-16}]
In the $t$-dimensional Euclidean $\mathbb{E}^{t}$,
we define the set $S_t$ of points as
$$
S_t=\left\{\left(\underbrace{0,0,\cdots,0}_{i-1},\frac{1}{\sqrt{2}},\underbrace{0,0,\cdots,0}_{j-i-1},\frac{1}{\sqrt{2}},\underbrace{0,0,\cdots,0}_{t-j}\right) \in \mathbb{E}^t : 1 \leq i < j \leq t \right\}.
$$
Note that $|S_t|={t\choose 2}$.
Let $\psi_t$ be a bijection from $S_t$
to the edge set $\{v_iv_j\,|\,1\leq i<j\leq t\}$ of a complete graph $K_{t}$
with vertex set $\{v_i : 1 \leq i \leq t\}$
such that
$$
\psi_t\left(\left(\underbrace{0,0,\cdots,0}_{i-1},\frac{1}{\sqrt{2}},\underbrace{0,0,\cdots,0}_{j-i-1},\frac{1}{\sqrt{2}},\underbrace{0,0,\cdots,0}_{t-j}\right)\right)=
v_iv_j.
$$
Note that
an $r$-coloring of $S_t$ determines an $r$-coloring of the edges of $K_{t}$.
\medskip

\noindent
$(1)$
We take the set $S_{r+4}$ and an $r$-coloring of $K_{r+4}$ as above.
By Theorem \ref{th-2-15} (1),
there is a rainbow $K_3$ or a monochromatic $C_4$
in the $r$-colored $K_{r+4}$.
\begin{itemize}
\item[] $(a)$ Suppose that there is a rainbow $K_3$,
say $v_1v_2v_3v_1$ without loss of generality.
Then, the three edges $v_1v_2$, $v_2v_3$ and $v_3v_1$ have the different colors.
The points in $\mathbb{E}^{r+4}$
corresponding to the edges $v_1v_2$, $v_2v_3$ and $v_3v_1$ through $\psi_{r+4}^{-1}$
are $A = \Big(\frac{1}{\sqrt{2}},\frac{1}{\sqrt{2}},\underbrace{0,\cdots,0}_{r+2}\Big)$,
$B = \Big(0,\frac{1}{\sqrt{2}},\frac{1}{\sqrt{2}},\underbrace{0,\cdots,0}_{r+1}\Big)$
and $C = \Big(\frac{1}{\sqrt{2}},0,\frac{1}{\sqrt{2}},\underbrace{0,\cdots,0}_{r+1}\Big)$,
respectively.
Since $|AB|=|BC|=|AC|=1$,
the three points $A,B,C$ form a rainbow equilateral triangle in $S_{r+4}\subseteq \mathbb{E}^{r+4}$.

\item[] $(b)$ Suppose that there is a monochromatic $C_4$,
say $v_1v_2v_3v_4v_1$ without loss of generality.
The points in $\mathbb{E}^{r+4}$
corresponding to the edges $v_1v_2$, $v_2v_3$, $v_3 v_4$ and $v_4v_1$ through $\psi_{t+4}^{-1}$
are
$A=\Big(\frac{1}{\sqrt{2}},\frac{1}{\sqrt{2}},\underbrace{0,\cdots,0}_{r+2}\Big)$,
$B=\Big(0,\frac{1}{\sqrt{2}},\frac{1}{\sqrt{2}},\underbrace{0,\cdots,0}_{r+1}\Big)$,
$C=\Big(0,0,\frac{1}{\sqrt{2}},\frac{1}{\sqrt{2}},\underbrace{0,\cdots,0}_{r}\Big)$ and
$D=\Big(\frac{1}{\sqrt{2}},0,0,\frac{1}{\sqrt{2}},\underbrace{0,\cdots,0}_{r}\Big)$,
respectively.
Since $|AB|=|BC|=|CD|=|AD|=1$ and $|AC|=|BD|=\sqrt{2}$,
the four points $A,B,C,D$ form a monochromatic
unit square in $S_{r+4} \subseteq \mathbb{E}^{r+4}$.

\end{itemize}

\noindent
$(2)$
We take the set $S_{r+3}$ and an $r$-coloring of $K_{r+3}$ as above.
By Theorem \ref{th-2-15} (2),
there is a rainbow $K_3$ or a monochromatic $P_4$
in the $r$-colored $K_{r+3}$.
\begin{itemize}
\item[] $(a)$ If there is a rainbow $K_3$,
then by the same was as in (1)-(a),
we can find a rainbow equilateral triangle in $S_{r+3} \subseteq \mathbb{E}^{r+3}$.

\item[] $(b)$
Suppose that there is a monochromatic $P_4$,
say $v_1v_2v_3v_4$ without loss of generality.
The points in $\mathbb{E}^{r+3}$
corresponding to the edges $v_1v_2$, $v_2v_3$ and $v_3v_4$ through $\psi_{r+3}^{-1}$
are
$A=\big(\frac{1}{\sqrt{2}},\frac{1}{\sqrt{2}},\underbrace{0,\cdots,0}_{r+2}\big)$,
$B=\big(0,\frac{1}{\sqrt{2}},\frac{1}{\sqrt{2}},\underbrace{0,\cdots,0}_{r+1}\big)$ and
$C=\big(0,0,\frac{1}{\sqrt{2}},\frac{1}{\sqrt{2}},\underbrace{0,\cdots,0}_{r}\big)$,
respectively.
Since $|AB|=|BC|=1$ and $|AC|=\sqrt{2}$,
the three points $A,B,C$ form a monochromatic
$(1,1,\sqrt{2})$-triangle in $S_{r+3} \subseteq \mathbb{E}^{r+3}$.

\end{itemize}

\end{proof}

\end{document}